\documentclass[12pt]{amsart}
\usepackage{amsfonts,amscd,latexsym,amsmath,amssymb,cancel,enumerate,mathrsfs,mathabx}
\usepackage[unicode]{hyperref}
\theoremstyle{plain}
\newtheorem{master}{Master}[section]
\newtheorem{prop}[master]{Proposition}
\newtheorem{thm}[master]{Theorem}
\newtheorem{fact}[master]{Fact}
\newtheorem{lem}[master]{Lemma}

\newtheorem{question}[master]{Question}
\newtheorem{claim}[master]{Claim}
\theoremstyle{definition}
\newtheorem{defin}[master]{Definition}
\newtheorem{constr}[master]{Construction}
\newtheorem{observation}[master]{Observation}
\newtheorem{example}[master]{Example}
\theoremstyle{remark}
\newtheorem{remark}[master]{Remark}
\numberwithin{equation}{section}

\newcommand{\Rea}{\mathbb{R}}
\newcommand{\Nat}{\mathbb{N}}
\newcommand{\Rat}{\mathbb{Q}}

\begin{document}
\title[Metric approximation of groups and their ultraproducts]{Metric topological groups: their metric approximation and metric ultraproducts}
\author[M. Doucha]{Michal Doucha}
\address{Laboratoire de Math\' ematiques de Besan\c con\\Universit\' e de Franche-Comt\' e\\France} 
\address[current address:]{Institute of Mathematics CAS, \v Zitn\' a 25, 115 67 Prague, Czech Republic}
\email{doucha@math.cas.cz}
\begin{abstract}
We define a metric ultraproduct of topological groups with left-invariant metric, and show that there is a countable sequence of finite groups with left-invariant metric whose metric ultraproduct contains isometrically as a subgroup every separable topological group with left-invariant metric.

In particular, there is a countable sequence of finite groups with left-invariant metric such that every finite subset of an arbitrary topological group with left-invariant metric may be approximated by all but finitely many of them.

We compare our results with related concepts such as sofic groups, hyperlinear groups and weakly sofic groups.
\end{abstract}
\maketitle
\section*{Introduction}
It is a major open problem whether all discrete groups are sofic, i.e. whether all discrete groups can be metrically approximated, in a certain sense, by finite permutation groups with the Hamming distance. On the other hand, when one wants to approximate metric groups, say with bi-invariant distance, it is clear finite permutation groups with the Hamming distance cannot serve for that purpose, e.g. the group of integers with the standard metric cannot be approximated by them. Since the introduction of sofic groups, many other classes of groups, defined in a similar manner as groups metrically approximable by certain class of `basic metric groups', appeared in the literature. Most notably the hyperlinear groups, formally introduced by R\v adulescu in \cite{Rad}, that are directly connected to the Connes' embedding conjecture for group von Neumann algebras (\cite{Con}). However, let us also mention linearly sofic groups introduced by Arzhantseva and P\v aunescu in \cite{ArPa}, $F_c$-approximable groups introduced by Thom in \cite{Thom}, and weakly sofic groups introduced by Glebsky and Rivera in \cite{GlRi} (see also \cite{Gl}). Thom in \cite{Thom} showed that the Higman's group is not $F_c$-approximable, however for all other classes this is unknown.

In this paper, we consider metric approximation by finite groups with left-invariant metrics (that do not have to be bi-invariant). We shall show that in this case we can prove a positive result. 
\begin{thm}\label{main1}
There exists a countable sequence $(G_n)_n$ of finite groups with left-invariant metric such that any finite subset of any topological group with left-invariant metric can be metrically approximated by all but finitely many $G_n$'s.
\end{thm}
We refer to the last section, where the theorem is proved, for a precise formulation and definition of approximation.

It is common in the area of group approximations to work with metric ultraproducts of metric groups. Indeed, being $\mathcal{C}$-approximable for a certain class $\mathcal{C}$ of metric groups (with bi-invariant metric) is equivalent with being embeddable as a subgroup into a metric ultraproduct of groups from $\mathcal{C}$. Metric ultraproducts of metric groups have been defined only for groups with bi-invariant metric. Here we generalize the notion and define a metric ultraproduct of arbitrary topological groups with left-invariant metric and obtain the following theorem.
\begin{thm}\label{main2}
There exists a countable sequence $(G_n)_n$ of finite groups with left-invariant metric whose metric ultraproduct contains isometrically an arbitrary separable topological group with left-invariant metric.
\end{thm}
We note that although being $\mathcal{C}$-approximable and being embeddable into metric ultraproduct of groups from $\mathcal{C}$ is rather easily checked to be equivalent when $\mathcal{C}$ contains just groups with bi-invariant metric, it is not the case in our general situation. The proof of Theorem \ref{main2} is substantially more involved than the proof of Theorem \ref{main1}. Indeed, the tricky issue with metric ultraproducts of groups with left-invariant metric is that in some cases the ultraproduct collapses to a trivial group, so one has to choose the sequence $(G_n)_n$ carefully.
\section{Definitions and preliminaries}
\subsection{Norms and metrics on groups}
Let $G$ be a group. A \emph{norm} (or a \emph{length function}) on $G$ is a function $\lambda:G\rightarrow \Rea_0^+$ with the following properties:
\begin{itemize}
\item $\lambda(x)=\lambda(x^{-1})$ for every $x\in G$,
\item $\lambda(x\cdot y)\leq \lambda(x)+\lambda(y)$ for every $x,y\in G$,
\item $\lambda(x)=0$ iff $x=1_G$.
\end{itemize}
$\lambda$ satisfying only the right-to-left implication of the last condition is called a \emph{seminorm}.

A (semi)norm $\lambda$ on $G$ satisfying $\lambda(g^{-1}\cdot h\cdot g)=\lambda(h)$ for every $g,h\in G$ is called \emph{conjugacy-invariant}.

Recall that a (pseudo)metric $d$ on the group $G$ is \emph{left-invariant} if $d(g\cdot x,g\cdot y)=d(x,y)$ for every $g,x,y\in G$. Right-invariance and bi-invariance are defined analogously.

There is a one-to-one correspondence between norms and left-invariant metrics (and analogously between seminorms and left-invariant pseudometrics). Indeed, given a left-invariant metric $d$, the formula $\lambda_d(x):=d(x,1_G)$ gives a norm on $G$; and conversely, given a norm $\lambda$ on $G$, the formula $d_\lambda(x,y):=\lambda(x^{-1}\cdot y)$ gives a left-invariant metric. 

Moreover, if the metric $d$ was bi-invariant, the the formula above gives a conjugacy-invariant norm. Conversely, if the norm $\lambda$ was conjugacy-invariant, then the formula above gives a bi-invariant metric.

It turns out it is more convenient for us to work with norms rather than metrics, so we will do so in the sequel.\\

It follows that (semi)norms on groups define a topology there. However, the topology on a group $G$ determined by some (semi)norm $\lambda$ on $G$ does not in general make it a topological group; i.e. the group operations are not automatically continuous. The following is a necessary and sufficient condition on a (semi)norm to make the group operations continuous. We leave the verification to the reader.
\begin{fact}
Let $\lambda$ be a (semi)norm on a group $G$. Then $G$ with the inherited topology is a topological group if and only if for every $x\in G$ and every $\epsilon>0$ there exists $\delta>0$ such that $\forall y\in G (\lambda(y)<\delta\Rightarrow \lambda(x^{-1}\cdot y\cdot x)<\varepsilon)$; in other words, the function $y\to \lambda(x^{-1}\cdot y\cdot x)$ is continuous at $1_G$.
\end{fact}
We shall call such (semi)norms \emph{continuous (semi)norms}. Note that when a (semi)norm is conjugacy-invariant then it is continuous. We remark that in literature, a norm being continuous often means that it is continuous with respect to some given topology on the group. Here however, the only group topologies we consider are those given by some norms, resp. pseudonorms.\\

Recall that when $f:X\rightarrow Y$ is a function between metric spaces $X$ and $Y$ which is continuous at the point $x\in X$, then a \emph{modulus of continuity} of $f$ at $x$ is a function $\omega:\left[0,\infty\right)\rightarrow \left[0,\infty\right)$ continuous at $0$ and vanishing there which quantitatively measures this continuity (of $f$ at $x$). That is, we have $d_Y(f(x),f(y))\leq \omega(d_X(x,y))$. Clearly, a modulus of continuity for a given function at a given point is not unique, however one can always take the `minimal one' by defining $\omega(r)=\sup\{d_Y(f(x),f(y)):y\in X, d_X(x,y)\leq r\}$. We shall use this notion in the context of normed groups.

\begin{defin}
Let $G$ be a group equipped with a continuous (semi)norm $\lambda$. We say that the functions $(\Gamma_x^G)_{x\in G}$, where $\Gamma_x^G: \left[0,\infty\right)\rightarrow \left[0,\infty\right)$ for every $x\in G$, are \emph{moduli of continuity}, or \emph{MOC}, for $G$ if for every $x\in G$:

\begin{itemize}
\item $\Gamma_x^G(r)\to_{r\to 0} 0$ and $\Gamma_x^G(0)=0$;
\item $\Gamma_x^G(r)\geq r$ for every $r\geq 0$;
\item for every $g\in G$ we have $\lambda(x^{-1}\cdot g\cdot x)\leq \Gamma_x^G(\lambda(g))$;
\item $\Gamma_x^G=\Gamma_{x^{-1}}^G$.
\end{itemize}

When considering a single element $x\in G$, we say that $\Gamma_x^G\in (\Gamma_x^G)_{x\in G}$ is a modulus of continuity (or MOC) for $x$ in $G$.
\end{defin}

We note that in \cite{DiGa}, in this context of groups with norms (resp. left-invariant metrics), these moduli are called scales.\\

For a given group $G$ with a continuous (semi)norm $\lambda$, moduli of continuity are not determined uniquely. However, it is again possible to consider the minimal moduli: for $x\in G$ and $r\in\left[0,\infty\right)$ set $$\Gamma_x^G=\max\{r,\sup\{\lambda(x^\varepsilon\cdot g\cdot x^{-\varepsilon}):g\in G,\lambda(g)\leq r,\varepsilon\in\{1,-1\}\}.$$ Note that such a MOC satisfies additionally
\begin{itemize}
\item $\Gamma_x^G(r)\leq 2\lambda(x)+r$.
\end{itemize}
Although we shall not always work with the minimal moduli, unless stated otherwise, $\Gamma_x^G$ will denote the minimal MOC for $x\in G$ in $G$.
\begin{example}
Let $(G,\lambda)$ be a normed group. Then $\lambda$ is conjugacy-invariant if and only if the minimal moduli $(\Gamma_x^G)_{x\in G}$ are constant functions, i.e. $\Gamma_x^G(r)=r$ for every $x\in G$ and $r\in\left[0,\infty\right)$.
\end{example}

The reason to work with MOC, even though they are not unique, is to control the `uniformity' of embeddings between normed groups. Suppose that $(G_1,\lambda_1)\subseteq (G_2,\lambda_2)\subseteq \ldots$ is an increasing sequence of groups with continuous norms. Then the direct union $(G,\lambda)$, where $G=\bigcup_n G_n$ and $\lambda=\bigcup_n \lambda_n$, is not in general a group with a continuous norm, i.e. the continuity of $\lambda=\lim_n \lambda_n$ may be lost in the limit. The reason for that is that when $\Gamma_x^i$ is a modulus of continuity for some $x\in G_i$ in $G_i$, it may no longer be a modulus of continuity for $x\in G_i\subseteq G_{i+1}$ in $G_{i+1}$. Later on, we will work with embeddings between normed groups that preserve some moduli of continuity in order to guarantee that norms on certain limit groups are still continuous.\\

We shall conclude this section with several other facts concerning normed groups.

First we want to recall the following geometric notion that will be useful later.
\begin{defin}
Let $(G,\lambda)$ be a normed topological group. We say that $\lambda$ is \emph{proper} if for every $r>0$ the set $\{g\in G:\lambda(g)\leq r\}$ is compact. In other words, $G$ with the induced metric is a proper metric space.

In case $(G,\lambda)$ is countable discrete, it means that for every $r>0$ the set $\{g\in G:\lambda(g)\leq r\}$ is finite.
\end{defin}

Second, we mention that if we have a group with a continuous seminorm we can always quotient to get a genuine norm on the quotient group.
\begin{fact}\label{makequotient}
Let $G$ be a group with a continuous seminorm $\lambda$. Then the set $N=\{g\in G:\lambda(g)=0\}$ is a closed normal subgroup, and $\lambda$ is constant on any left coset of $N$, thus it determines a continuous norm on $G/N$.
\end{fact}
\begin{proof}
$N$ is by the definition of the topology on $G$ closed. Since for any $g,h\in G$ we have $\lambda(g)=\lambda(g^{-1})$, $\lambda(g\cdot h)\leq \lambda(g)+\lambda(h)$ and $\lambda$ is continuous, it immediately follows that $N$ is a normal subgroup. Take any $x\in G$ and $g\in N$. We show that $\lambda(x)=\lambda(x\cdot g)$. We have $\lambda(x\cdot g)\leq \lambda(x)+\lambda(g)=\lambda(x)=\lambda(x\cdot g\cdot g^{-1})\leq \lambda(x\cdot g)+\lambda(g^{-1})=\lambda(x\cdot g)$.
\end{proof}

Finally, in order to persuade the reader that there are indeed a plethora of groups with continuous (semi)norms, let us mention the classical result of Birkhoff and Kakutani. It says that a group $G$ with topology $\tau$ is a first-countable topological group if and only if there exists a continuous seminorm on $G$ which induces the topology $\tau$ of $G$. Moreover, $G$ is Hausdorff if and only if the seminorm is a norm.
\subsection{Completeness in normed groups}\label{completeness_sect}
Now for a moment, we switch to continuous left-invariant (pseudo)metrics rather than (semi)norms, where by a \emph{continuous} left-invariant (pseudo)metric we mean a left-invariant (pseudo)metric whose associated (semi)norm is continuous. So assume we are given a group $G$ with a continuous left-invariant (pseudo)metric $d$. It is well known that a metric completion of $G$ with respect to $d$ need not to be a group, however it is always a semigroup. Indeed, it is an exercise to check that the multiplication operation extends to the metric completion; in other words, whenever $(x_n)_n$ and $(y_n)_n$ are Cauchy sequences in $G$, then $(x_n\cdot y_n)_n$ is a Cauchy sequence as well. On the other hand, the inverse operation might not extend to the completion since the sequence $(x_n)_n$ being Cauchy does not guarantee that the sequence of inverses $(x^{-1}_n)_n$ is also Cauchy. Consider for example $S_\infty$, the infinite permutation group of $\Nat$, with a left-invariant metric $d$ defined as $d(x,y)=\max\{1/n:x(n)\neq y(n)\}$. Completion of $S_\infty$ with respect to this metric is the semigroup of all injective mappings from $\Nat$ into $\Nat$.\\

However, there is another way how to canonically complete a group with a left-invariant metric. 
\begin{fact}\label{obs2}
Let $G$ and $d$ be as before. Consider the metric $D(x,y):=d(x,y)+d(x^{-1},y^{-1})$ and the completion of $G$ with respect to $D$. Then the group operations and the original metric $d$ extend to this completion. 
\end{fact}
We shall call it a \emph{Ra\v ikov metric completion} of $G$, since it precisely corresponds to the Ra\v ikov completion of a topological group. A normed/metric group $G$ whose Ra\v ikov metric completion coincides with $G$ is called \emph{Ra\v ikov metrically complete}. Note that the Ra\v ikov metric completion is nothing but adding limits for all Cauchy sequences $(x_n)_n\subseteq G$ such that the sequence of inverses $(x^{-1}_n)_n$ is also Cauchy.
\subsection{Free groups}
Finally, since we shall work with free groups often we recall some basic facts and fix some notation related to them here. Let $A$ be a non-empty set. Recall that the free group $F_A$ generated by $A$ is the free group having elements of $A$ as free generators. Consider the disjoint union $\{1\}\coprod A\coprod A^{-1}$ denoted by $\bar{A}$, where $A^{-1}$ is the set of formal inverses of $A$, i.e. $A^{-1}=\{a^{-1}:a\in A\}$. One can view the free group $F_A$ as the set of all reduced words over the alphabet $\bar{A}$. A word $w=w_1\ldots w_n$, where $w_1,\ldots,w_n\in\bar{A}$ is reduced if either $n=1$ and $w_1=1$, or there is no $i\leq n$ such that $w_i=1$ and $w_i=w_{i+1}^{-1}$. For any word (not necessarily reduced) $w$ over the alphabet $\bar{A}$, by $w'$ we denote the reduction of $w$, i.e. the unique reduced word obtained from $w$ by successively removing the pairs $w_i, w_{i+1}$, where $w_i=w_{i+1}^{-1}$, and letters $1$ from $w$ till it is reduced. In case this procedure leads to an empty word, we 
set $w'$ to be $1$. For any word $w$, by $|w|$ we denote the length of the word, i.e. the number of letters from alphabet used to make $w$.

Then the group multiplication of two reduced words $w_1$ and $w_2$ is defined to be $(w_1w_2)'$, i.e. concatenation of two words followed by reduction. The inverse of a reduced word $w_1\ldots w_n$ is the reduced word $w_n^{-1}\ldots w_1^{-1}$. The unit is the reduced word $1$.\\

We shall also use the following basic observation.
\begin{observation}\label{obs1}
Let $H$ be an at most countable group equipped with a (continuous) norm $\lambda$. Then there exists a (continuous) seminorm $\lambda'$ on $F_\infty$, the free group of countably many free generators, such that the quotient $F_\infty/N$, where $N=\{h\in F_\infty:\lambda'(h)=0\}$, is isometrically isomorphic to $(H,\lambda)$.
\end{observation}
Indeed, just pick some countable set of generators (with possible repetition) $(h_n)_n$. For each reduced word $w$ over the alphabet $\{1,h_n,h_n^{-1}:n\in\Nat\}$ denote by $w_H$ its evaluation in $H$, i.e. the group element of $H$ that corresponds to the natural evaluating of $w$ in $H$. Then we consider the free group freely generated by $(h_n)_n$ and define the seminorm $\lambda'$ by the formula $\lambda'(w)=\lambda(w_H)$ for any word $w$ over the alphabet $\{1,h_n,h_n^{-1}:n\in\Nat\}$.\\

\section{Normed ultraproducts of normed groups and group embeddings into them}
Metric ultraproducts of groups with bi-invariant metric, resp. conjugacy-invariant norms are well-known from the literature. We refer to the appendix in \cite{CaLu} for information about them. Let $\mathcal{M}$ be some class of groups equipped with bi-invariant metric/conjugacy-invariant norms. It is of great current interest which (discrete) groups can be embedded into a metric ultraproduct of groups from $\mathcal{M}$.

As already mentioned in the introduction, the most interesting cases are when $\mathcal{M}$ is the set of unitary groups of finite rank equipped with the Hilbert-Schmidt distance and when $\mathcal{M}$ is the set of finite permutation groups equipped with the normalized Hamming distance. The former are the hyperlinear groups and the latter are the sofic groups. We recall they were introduced by Gromov (\cite{Gro}). They are related to the Gottschalk's surjunctivity conjecture. The major open problem is whether every group is hyperlinear and sofic (we note that every sofic group is hyperlinear \cite{ElSz}). We refer the reader to the survey \cite{Pe-sh} and to the monograph \cite{CaLu} where these classes of groups are defined and metric ultraproducts of groups with bi-invariant metrics are treated.

Weakly sofic groups are $\mathcal{M}$-approximable groups, where $\mathcal{M}$ is the class of all finite groups with arbitrary bi-invariant metric. Weakly sofic groups as a generalization of sofic groups were introduced by Glebsky and Rivera in \cite{GlRi} (see also \cite{Gl}) as the existence of a non-weakly sofic group is equivalent to a certain conjecture about pro-finite topology on finitely generated free groups.

We also recall from the introduction the linear sofic groups introduced by Arzhantseva and Paunescu in \cite{ArPa}, which are groups approximable by general linear groups with the normalized rank distance. When $\mathcal{M}$ is the set of finite groups with a commutator-contractive bi-invariant metric, then such $\mathcal{M}$-approximable groups were called as $F_c$-approximable groups in \cite{Thom}. Finally, let us mention that when $\mathcal{M}$ consists of all finite groups with the trivial metric (i.e. taking only $\{0,1\}$ as values), then such groups were called LEF (locally embeddable into finite) by Gordon and Vershik (\cite{GoVe}) (similarly, $\mathcal{M}$ is LEA if it consists of finitely generated amenable groups with trivial metric).

So far, it has been widely open whether there are groups which are \emph{not} approximable by any such classes $\mathcal{M}$ mentioned. The only exceptions besides the rather simple case of LEF groups (or analogously LEA groups) is when $\mathcal{M}$ is $F_c$, as it was proved by Thom in \cite{Thom} that the Higman's group is not $F_c$-approximable.

\subsection{Definition of the metric ultraproduct}
Let now $(G_n,\lambda_n)_{n\in\Nat}$ be a sequence of general normed groups and fix some non-principal ultrafilter $\mathcal{U}$ on $\Nat$. We would like to define a metric/normed ultraproduct of them.  Before we proceed any further let us remark here that in this paper we consider only ultraproducts of countable sequences of groups, thus all ultrafilters are over $\Nat$. Also, whenever we say \emph{ultraproduct} we automatically mean an ultraproduct determined by a non-principal ultrafilter.

We begin with recalling some standard constructions of metric ultraproducts. At first, one takes the direct product $\prod_n G_n$. In order to define an ultraproduct norm there one has to restrict to a subgroup of the product of those elements whose coordinates have norm bounded by one common constant. That is, using a Banach space theory notation, let $(G_n)_{\ell_\infty}=\{(g_n)_n:\sup_n \lambda_n(g_n)<\infty)\}$. Let $\lambda_\infty$ be the supremum norm on $(G_n)_{\ell_\infty}$. Consider then the subgroup $N=\{(g_n)_n\in (G_n)_{\ell_\infty}:\lim_\mathcal{U} \lambda_n(g_n)=0\}$. If all the $\lambda_n$'s were conjugacy-invariant, then $N$ is a normal subgroup and the quotient $(G_n)_{\ell_\infty}/N=(G_n)_\mathcal{U}$ with the quotient norm is the metric ultraproduct of the sequence $(G_n,\lambda_n)_n$. 

Alternatively, one might equip $(G_n)_{\ell_\infty}$ with the ultraproduct seminorm $\lambda_\mathcal{U}$, where $\lambda_\mathcal{U}((g_n))=\lim_\mathcal{U} \lambda_n(g_n)$ and again consider the kernel $N=\{(g_n)_n\in (G_n)_{\ell_\infty}:\lambda_\mathcal{U} ((g_n))=0\}$. If the norm $\lambda_\mathcal{U}$ is continuous, $N$ will be a normal subgroup and we can take the quotient. Again, if all $\lambda_n$'s are conjugacy invariant then $\lambda_\mathcal{U}$ will be conjugacy-invariant as well, and thus continuous. So $N$ is a normal subgroup.\\

If not all $\lambda_n$'s are conjugacy-invariant then $\lambda_\mathcal{U}$ is an ultraproduct seminorm which however does not have to be continuous, thus $((G_n)_{\ell_\infty},\lambda_\mathcal{U})$ is not a topological group and the kernel subgroup does not have to be normal. In such a case, one has to restrict the subgroup $(G_n)_{\ell_\infty}\leq \prod_n G_n$ more. More precisely, we shall restrict to a subset of $(G_n)_{\ell_\infty}\leq \prod_n G_n$ (which will turn out to be a subgroup) of elements that obey some modulus of continuity. That is the content of the following definition.
\begin{defin}
Call an element $(g_n)_n\in (G_n)_{\ell_\infty}$ \emph{continuous in the ultraproduct} if 
\begin{multline}\label{contultprod}
\forall \varepsilon>0\; \exists\delta>0\; \exists A\in\mathcal{U}\text{ such that }\forall n\in A\; \forall h_n\in G_n\\ \text{ if }\lambda_n(h_n)\leq \delta \text{ then } \lambda_n(g_n^{-1}\cdot h_n\cdot g_n)<\varepsilon \text{ and }\lambda_N(g_n\cdot h_n\cdot g_n^{-1})<\varepsilon.
\end{multline}
\end{defin}
 
Equivalently, one can view elements that are continuous in the ultraproduct as follows. For each $(g_n)_n\in (G_n)_{\ell_\infty}$ take some corresponding sequence $(\Gamma_n)_n$ of moduli of continuity (provided they do exist), i.e $\Gamma_n$ is an MOC for $g_n$ in $G_n$. We take the ultralimit of this sequence of moduli, i.e. we define $\Gamma_\mathcal{U}(r)=\lim_\mathcal{U} \Gamma_n(r)$. If this ultralimit $\Gamma_\mathcal{U}$ is again a MOC (for $(g_n)_n$ in $(G_n)_{\ell_\infty}$), then $(g_n)_n$ is continuous in the ultraproduct. Conversely, if $(g_n)_n$ is continuous in the ultraproduct then there exists a sequence $(\Gamma_n)_n$ of moduli of continuity such that $\Gamma_n$ is an MOC for $g_n$ in $G_n$ and the ultralimit $\Gamma_\mathcal{U}$ is an MOC for $(g_n)_n$ in $(G_n)_{\ell_\infty}$.\\

Denote by $(G_n)_\mathcal{C}\leq (G_n)_{\ell_\infty}$ the subset of elements continuous in the ultraproduct. 
\begin{lem}
$(G_n)_\mathcal{C}$ is a subgroup of $(G_n)_{\ell_\infty}$. Moreover, if $\lambda_n$'s were conjugacy-invariant, then $(G_n)_\mathcal{C}=(G_n)_{\ell_\infty}$.
\end{lem}
\begin{proof}
If $(g_n)_n\in (G_n)_\mathcal{C}$ then by definition also $(g^{-1}_n)_n\in (G_n)_\mathcal{C}$, thus $(G_n)_\mathcal{C}$ is closed under taking inverses. Now pick some $(g_n)_n,(h_n)_n\in (G_n)_\mathcal{C}$. We show that $(g_n\cdot h_n)_n\in (G_n)_\mathcal{C}$. Take some $\varepsilon>0$ and we must find corresponding $A_\varepsilon\in \mathcal{U}$ and $\delta>0$ from the definition. By assumption, there are some $\delta'>0$ and $A_g\in\mathcal{U}$ such that for all $n\in A_g$ and $f_n\in G_n$ such that $\lambda_n(f_n)\leq \delta'$ we have $\lambda_n(g^\iota_n\cdot f_n\cdot g^{-\iota}_n)<\varepsilon$, for $\iota\in \{1,-1\}$. Similarly, by assumption, there are some $\delta>0$ and $A_h\in\mathcal{U}$ such that for all $n\in A_h$ and $f_n\in G_n$ such that $\lambda_n(f_n)\leq \delta$ we have $\lambda_n(h^\iota_n\cdot f_n\cdot h^{-\iota}_n)<\delta'$, for $\iota\in \{1,-1\}$. Now it is clear $A_\varepsilon=A_g\cap A_h$ and $\delta>0$ are as desired.

The moreover statement from the lemma is easy and left to the reader.
\end{proof}
 We consider the ultraproduct seminorm $\lambda$ on $(G_n)_\mathcal{C}$.
\begin{lem}
The ultraproduct seminorm $\lambda$ on $(G_n)_\mathcal{C}$ is continuous, thus the kernel subgroup is normal and we can quotient. 
\end{lem}
\begin{proof}
Indeed, take some $(g_n)_n\in (G_n)_\mathcal{C}$ and $\varepsilon>0$. By definition, there is some $B\in\mathcal{U}$ and $\delta>0$ such that for every $n\in B$ we have $\lambda_n(g_n^{-1}\cdot h\cdot g_n)<\varepsilon$ for every $h\in G_n$ such that $\lambda_n(h)<\delta$. Take now some $(h_n)_n\in (G_n)_\mathcal{C}$ such that $\lambda((h_n))<\delta$. We need to show that $\lambda((g_n)^{-1}\cdot (h_n)\cdot (g_n))<\varepsilon$. It suffices to find $A_\varepsilon\in\mathcal{U}$ such that for every $n\in A_\varepsilon$ we have $\lambda_n(g_n^{-1}\cdot h_n\cdot g_n)<\varepsilon$. Since $\lambda((h_n))<\delta$ there is some $C\in\mathcal{U}$ such that for every $n\in C$ we have $\lambda_n(h_n)<\delta$. Thus it suffices to take $A_\varepsilon=B\cap C$.
\end{proof}

We note that one typical element of $(g_n)_n\in (G_n)_\mathcal{C}$ is such that there is a single MOC $\Gamma$ such that $\Gamma$ is a MOC for $g_n$ in $(G_n,\lambda_n)$ for every $n$.
\subsection{Ra\v ikov metric completeness}
Finally, we make some observations regarding the Ra\v ikov metric completeness defined in the previous section. It is known that ultraproducts of normed vector spaces or groups with conjugacy-invariant norms are complete. A group with a norm cannot be always complete as noted in Subsection \ref{completeness_sect}. However, they may be Ra\v ikov metrically completed as mentioned in Fact \ref{obs2}.
\begin{lem}
A metric ultraproduct of normed groups $(G_n,\lambda_n)$ is Ra\v ikov metrically complete, regardless of whether $G_n$'s were Ra\v ikov metrically complete.
\end{lem}
\begin{proof}
Suppose we have a sequence (of sequences) $((g_{n,m})_n)_m\subseteq (G_n)_\mathcal{C}$ of elements, resp. representatives from the equivalence classes, from the metric ultraproduct such that both the sequence and the sequence of its inverses are Cauchy. We shall show that the limit is in $(G_n)_\mathcal{C}$. The limit is constructed as in the case of normed vector spaces or groups with conjugacy-invariant norms. That is, let $(A_n)_n$ be a strictly decreasing sequence of sets from the ultrafilter $\mathcal{U}$ such that $\bigcap_n A_n=\emptyset$, and $(k_n)_n$ a strictly increasing sequence of natural numbers such that for every $m$ and every $i,j\geq k_m$ we have $$\forall n\in A_m (\lambda_n(g_{n,i}^{-1}\cdot g_{n,j})<1/2^m\wedge \lambda_n(g_{n,i}\cdot g_{n,j}^{-1})<1/2^m).$$
The limit sequence $(h_n)_n$ is defined so that for all $n\notin A_1$ we have $h_n=1$ and for any $m$ and $n\in A_m\setminus A_{m+1}$ we have $h_n=g_{n,k_m}$. We claim that $(h_n)_n\in (G_n)_\mathcal{C}$ and that it is the limit of $((g_{n,m})_n)_m$, while $(h^{-1}_n)_n$ is the limit of $((g^{-1}_{n,m})_n)_m$. The latter is verified as in the classical case of groups with conjugacy-invariant norms, so we only check the former, i.e. that $(h_n)_n\in (G_n)_\mathcal{C}$.

By definition, we must check that for every $\varepsilon>0$ there are $\delta>0$ and $A_\varepsilon\in\mathcal{U}$ such that for every $n\in A_\varepsilon$ and $f_n\in G_n$ with $\lambda_n(f_n)<\delta$ we have $\lambda_n(h_n^{-1}\cdot f_n\cdot h_n)<\varepsilon$ and $\lambda_n(h_n\cdot f_n\cdot h_n^{-1})<\varepsilon$. Pick $l$ such that $1/2^l<\varepsilon/3$. Since $(g_{n,k_l})_n\in (G_n)_\mathcal{C}$ we have that there is some $A'\in\mathcal{U}$ and some $\delta>0$ such that for every $n\in A'$ and every $f_n\in G_n$ with $\lambda_n(f_n)<\delta$ we have
\begin{equation}\label{eq1}
\lambda_n(g_{n,k_l}^{-1}\cdot f_n\cdot g_{n,k_l})<\varepsilon/3,\quad
\lambda_n(g_{n,k_l}\cdot f_n\cdot g_{n,k_l}^{-1})<\varepsilon/3.
\end{equation}
Set $A_\varepsilon=A'\cap A_l\in\mathcal{U}$. For any $n\in A_\varepsilon\subseteq A_l$ and any $i>k_l$ we thus have 
\begin{equation}\label{eq2}
\lambda_n(g_{n,i}^{-1}\cdot g_{n,k_l})<\varepsilon/3\wedge \lambda_n(g_{n,i}\cdot g_{n,k_l}^{-1})<\varepsilon/3.
\end{equation}
Putting \eqref{eq1} and \eqref{eq2} together we get that for every $n\in A_\varepsilon$ and every $f_n\in G_n$ with $\lambda_n(f_n)<\delta$ we have $$\lambda_n(h_n^{-1}\cdot f_n\cdot h_n)\leq \lambda_n(h_n^{-1}\cdot g_{n,k_l})+\lambda_n(g_{n,k_l}^{-1}\cdot f_n\cdot g_{n,k_l})+\lambda_n(g_{n,k_l}^{-1}\cdot h_n)<$$ $$\varepsilon/3+\varepsilon/3+\varepsilon/3=\varepsilon.$$
Analogous inequalities give that $$\lambda_n(h_n\cdot f_n\cdot h_n^{-1})<\varepsilon,$$ and so we are done.
\end{proof}
\subsection{Some pathological examples}
We finish this section by presenting some pathological examples which show that metric ultraproducts of groups with general continuous norms are rather delicate. We show, as mentioned in the introduction, that a metric ultraproduct of normed topological groups may collapse to a trivial group. Also, we show that for some normed topological groups it may happen that their metric ultrapower is the group itself.

Let us start with the former.
\begin{lem}
There exists a sequence of non-trivial normed topological groups $(G_n,\lambda_n)$ such their metric ultraproduct, over any non-principal ultrafilter, is a trivial group.
\end{lem}
\begin{proof}
For every $n\in\Nat$, let $G_n$ be $\mathbb{F}_2$, the free group on two free generators. Let $|\cdot|$ be the canonical length function on $\mathbb{F}_2$, i.e. identifying $\mathbb{F}_2$ with the set of reduced words over the alphabet $\{a,b,a^{-1},b^{-1}\}$, $|x|$, for $x\in \mathbb{F}_2$, is the length of $x$ as a word. Let $\lambda_n$ be the rescaling $|\cdot|/n$. We claim this sequence is as desired. Fix any non-principal ultrafilter $\mathcal{U}$ on $\Nat$. Suppose there exists a non-trivial element $(g_n)_n\in (G_n)_\mathcal{U}$ in the metric ultraproduct, or rather its representative from $(G_n)_{\ell_\infty}$. Since $\lambda_\mathcal{U}((g_n))>0$, there exist $\varepsilon>0$ and $A\in\mathcal{U}$ such that for all $n\in A$, $\lambda_n(g_n)>\varepsilon$. Then we claim that there are no $\delta>0$ and $B\in\mathcal{U}$ such that for all $n\in B$ and $h_n\in G_n$ with $\lambda_n(h_n)<\delta$ we have $\lambda_n(g^{-1}_n\cdot h_n\cdot g_n)<\varepsilon$, thus violating the condition that $(g_n)_n$ is continuous in the ultraproduct. Suppose otherwise and fix corresponding $\delta>0$ and $B\in\mathcal{U}$. We may suppose that $B\subseteq A$. Pick $n\in B$ such that $1/n<\delta$. Recall that $g_n$ is some reduced word $w_1\ldots w_m$ over the alphabet $\{a,b,a^{-1},b^{-1}\}$. Take $x\in\{a,b\}$ such that $x\neq w_1$ and $x\neq w_1^{-1}$. We have that $\lambda_n(x)=1/n<\delta$. However, $\lambda_n(g_n^{-1}\cdot x\cdot g_n)>2\varepsilon+1/n$. Indeed, by assumption there is no cancelation in the word $w=w^{-1}_m\ldots w^{-1}_1 x w_1\ldots w_m$, thus $g_n^{-1}\cdot h\cdot g_n$ corresponds to the reduced word $w$. This finishes the proof.
\end{proof}

Next, we present an example of a normed topological group whose metric ultrapower is equal to the original group itself.
\begin{lem}
Consider the group $S_\infty$ of all permutations of $\Nat$ with the norm $\lambda(p)$, for $p\in S_\infty$, defined as $\max\{1/n:p(n)\neq n\}$, which was already considered in this section. Then its metric ultrapower (over any ultrafilter on $\Nat$) is equal to $S_\infty$ itself.
\end{lem}
 \begin{proof}
Let us start with an observation.
\begin{observation}\label{obs3}
Take any $p\in S_\infty$. For any $n$ we want to compute the $\delta>0$ such that whenever $\lambda(s)<\delta$ then we have $\lambda(p^{-1}\cdot s\cdot p)< 1/n$, and conversely that there exists $s\in S_\infty$ such that $\lambda(s)\geq \delta$ and $\lambda(p^{-1}\cdot s\cdot p)\geq 1/n$. Set $m=\max\{p(l):l\leq n\}$. We claim that we may take $\delta=1/m$. Indeed, suppose that for some $s\in S_\infty$ we have $\lambda(s)<1/m$. Then $s\upharpoonright \{1,\ldots,m\}=\mathrm{id}$. It follows that $p^{-1}\cdot s\cdot p\upharpoonright \{1,\ldots,n\}=\mathrm{id}$, thus $\lambda(p^{-1}\cdot s\cdot p)<1/n$. Conversely, let $m'=p(n)\leq m$. Let $s\in S_\infty$ be arbitrary with the property that $s(m')>m$. Then $\lambda(s)\geq 1/m$ and $p^{-1}\cdot s\cdot p\upharpoonright \{1,\ldots,n\}\neq\mathrm{id}$, thus $\lambda(p^{-1}\cdot s\cdot p)\geq 1/n$.\\
\end{observation}

Now consider the ultrapower of $S_\infty$ with respect to some non-principal ultrafilter $\mathcal{U}$ (on $\Nat$). Let $(p_n)_n$ be some sequence representing an element of the ultrapower. We claim that $$\forall n\: \exists m\: \exists A\in \mathcal{U}\:\forall i\in A\:\forall l\leq n\:(p_i(l)\leq m).$$ Otherwise, we would get that there is $n$ such that for every $m$ there is $A\in\mathcal{U}$ such that for every $i\in A$ we have $p_i(n)>m$. Note that the preceding formula is not a formal negation of the formula above, however it is equivalent to it. However, it follows from Observation \ref{obs3} that such a sequence is not continuous in the ultrapower. The same argument gives that $$\forall n\: \exists m\: \exists A\in \mathcal{U}\:\forall i\in A\:\forall l\leq n\:(p^{-1}_i(l)\leq m).$$
Now it follows that for any $n$ there is $A_n\in\mathcal{U}$ and $s_n\in S_\infty$ such that for every $i\in A$ and every $l\leq n$ we have $p_i(l)=s_n(l)$ and $p^{-1}_i(l)=s^{-1}_n(l)$. A straightforward argument gives that $(s_n)_n$ converges to some $s\in S_\infty$, and that $(p_n)_n$ is equal to the constant sequence consisting of $s$ in the ultrapower.
\end{proof}
\section{Proof of the main theorems}
In the last section, we prove Theorems \ref{main1} and \ref{main2}. The meaning of Theorem \ref{main2} is now clear after we have defined metric ultraproducts of normed groups in the previous section. We precisely restate Theorem \ref{main1} here. We start with a definition first.
\begin{defin}
Let $(G,\lambda)$ and $(H,\rho)$ be normed groups, and let $F\subseteq G$ be a finite subset and $\varepsilon>0$ arbitrary. We say that $\phi:F\rightarrow H$ is an $\varepsilon$-homomorphism if 
\begin{itemize}
\item $\rho(\phi(g\cdot h)^{-1}\cdot \phi(g)\cdot\phi(h))<\varepsilon$, for all $g,h\in F$ such that $g\cdot h\in F$;
\item $|\rho(\phi(g))-\lambda(g)|<\varepsilon$ for all $g\in F$.
\end{itemize}
\end{defin}
\begin{thm}\label{main1_restate}
There exists a countable sequence $(G_n,\lambda_n)_n$ of finite normed groups such that for any normed topological group $(H,\rho)$, in particular for any discrete group, and any $\varepsilon>0$ and any finite subset $F\subseteq H$ there exists $i_0$ such that for all $i\geq i_0$ there is an $\varepsilon$-homomorphism $\phi:F\rightarrow G_i$.

Moreover, we may require that for every $f\in F$, $\Gamma_{\phi(f)}^{G_i}\leq 2\Gamma_f^H+\varepsilon\mathrm{id}$.
\end{thm}
The rest is devoted to the proofs of the main theorems. We prove Theorem \ref{main2} and then show how Theorem \ref{main1_restate} follows.

Again, we need some definitions before we can continue.
\begin{defin}
Let $G$ be a finitely generated group. Let $A\subseteq G$ be some finite symmetric subset, i.e. $A=A^{-1}=\{a^{-1}:a\in A\}$, containing the unit $1_G$ and generating $G$. Consider a function $\lambda':A\rightarrow \Rea$ satisfying the following conditions:
\begin{itemize}
\item For $x\in A$, $\lambda'(x)=0$ if and only if $x=1_G$;
\item For any $x\in A$, $\lambda'(x)=\lambda'(x^{-1})$.
\end{itemize}
Then we call $\lambda'$ a partial pre-norm. If $\lambda'$ additionally satisfies condition
\begin{itemize}
\item For any $x_1,\ldots,x_n\in A$ such that $x_1\cdot \ldots\cdot x_n\in A$, $\lambda'(x\cdot\ldots\cdot x_n)\leq \sum_{i=1}^n \lambda'(x_i)$
\end{itemize}
then we call $\lambda'$ a partial norm.
\end{defin}
\begin{constr}
Let $G$ be a group, $A$ a finite symmetric subset containing the unit and generating $G$, and let $\lambda':A\rightarrow \Rea$ be a partial pre-norm. Then the following formula defines a norm $\lambda$ on $G$. For any $x\in G$ we set $$\lambda(x)=\min\{\lambda'(x_1)+\ldots+\lambda'(x_n):x_1,\ldots,x_n\in A, x=x_1\cdot\ldots\cdot x_n\}.$$

Indeed, it immediately follows from the definition that for any $x,y\in G$ we have $\lambda(x\cdot y)\leq \lambda(x)+\lambda(y)$. Since $\lambda'$ was a symmetric function vanishing at $1_G$ we get that also $\lambda$ is symmetric and vanishes at $1_G$.

We shall call such $\lambda$ \emph{finitely generated}.

Moreover, if $G$ is a finitely generated free group then observe that if $\lambda'$ is a partial norm then $\lambda$ extends $\lambda'$, and $\lambda$ is proper.
\end{constr}

Now suppose we have finitely many finitely generated free groups $F_1,\ldots,F_n$. For each $i\leq n$, suppose that $F_i$ is freely generated by $x_{i,1},\ldots,x_{i,n_i}$. Suppose also that for each $i\leq n$ there is a pre-norm $\lambda'_i:$ defined on some finite symmetric $A_i\subseteq F_i$ that contains $\{1,x_{i,1},\ldots,x_{i,n_i}\}$, which thus defines some norm $\lambda_i$ on $F_i$. $(F_i,\lambda_i)$ is a discrete normed group, thus a topological group. For any $i\leq n$ and $j\leq n_i$ denote by $\Gamma_i^j$ the minimal MOC for $x_{i,j}$ in $F_i$. That is, for any $r\in\left[0,\infty\right)$ define $$\Gamma_i^j(r)=\max\{r,\sup\{\lambda_i(x_{i,j}^\varepsilon\cdot g\cdot x_{i,j}^{-\varepsilon}):g\in F_i,\lambda_i(g)\leq r,\varepsilon\in\{1,-1\}\}.$$

Now consider the free product $F=F_1\ast\ldots\ast F_n$. We would like to define a finitely generated norm $F$ which extends the particular norms on $F_i$'s in such a way that the minimal moduli of continuity of the free generators in $F$ are `close' to the minimal moduli of the generators in the appropriate $F_i$'s. Before doing so, we need the following definition, first used in \cite{DiGao} and implicitly present already in \cite{SiUs}.
\begin{defin}[Match]
Let $A$ be some symmetric alphabet, i.e. if $a\in A$, then also its formal inverse $a^{-1}$ belongs to $A$. Let $w=w_m\ldots w_{m+n}$ be some word over $A$, for technical reasons enumerated by an arbitrary interval of natural numbers. Denote by $J$ that interval, i.e. $J=\{m,\ldots,m+n\}$. A \emph{match} on $J$ for $w$ is a bijection $\rho:J\rightarrow J$ such that
\begin{itemize}
\item $\rho\circ\rho=\mathrm{id}_J$, i.e. for every $i\in J$ we have $\rho\circ\rho(i)=i$,
\item for no $i,j\in J$ we have $i<j<\rho(i)<\rho(j)$,
\item if $\rho(i)\neq i$, for some $i\in J$, then $w_i=w^{-1}_{\rho(i)}$.
\end{itemize}
\end{defin}
Notice that for any match $\rho$ on $J$ for a word $w$ enumerated by $J$ and for any $i\in J$ such that $i<\rho(i)$, we have that $\rho\upharpoonright [i+1,\ldots,\rho(i)-1]$ is a match on $[i+1,\ldots,\rho(i)-1]$ for the corresponding subword of $w$.

Also, if $J$ and $K$ are disjoint intervals such that $\max J +1=\min K$, and $\rho_J$ is a match on $J$ for some word $w_J$ while $\rho_K$ is a match on $K$ for some word $w_K$, then $\rho_J\cup\rho_K$ is a match on $J\cup K$ for $w_J w_K$.\\

The reader should view a match $\rho$ for some word $w$ as a way how to build $w$ from its subwords by means of concatenation and conjugation. For example, for a word $w=a^{-1}bca$ and a match $\rho(1)=4,\rho(4)=1,\rho(2)=2,\rho(3)=3$ for $w$ one sees $w$ as being built first by concatenating letters $b$ and $c$ to obtain the word $bc$, and then conjugating $bc$ to obtain $a^{-1}bca$.

Now we are ready to state the proposition.
\begin{prop}\label{prop1}
There exists a finitely generated norm $\lambda$ on $F$ satisfying
\begin{itemize}
\item that for any $i\leq n$, $\lambda\upharpoonright F_i=\lambda_i$, i.e. $\lambda$ extends $\lambda_i$,
\item for every $i\leq n$ and $j\leq n_i$ and any $y\in F$, $\varepsilon\in\{1,-1\}$ we have $\lambda(x_{i,j}^\varepsilon\cdot y\cdot x_{i,j}^{-\varepsilon})\leq 2\Gamma_i^j(\lambda(y))$, i.e. $2\Gamma_i^j$ is a MOC for $x_{i,j}$ (and $x_{i,j}^{-1}$) in $F$.
\end{itemize}
\end{prop}
\begin{remark}
We stress the importance of the second item in the proposition, i.e. that there are moduli of the free generators in the free product that are close to the minimal moduli of the free generators in the original free groups.
\end{remark}
\begin{proof}
The norm $\lambda$ will be constructed in three steps. In the first step, we shall construct a finitely generated norm on $F$ that extends each $\lambda_i$. However, this norm will not yet satisfy the second condition from the statement of the proposition. In the second step, we shall modify the norm from the first step so that it still extends $\lambda_i$'s and moreover the minimal moduli satisfy the second condition. While doing so, we shall however break the condition that the norm is finitely generated. That will be fixed in the last third step.\\

\noindent {\bf Step 1.} First, set $B'=\bigcup_{i=1}^n A_i$ and $\sigma'=\bigcup_{i=1}^n \lambda'_i$. We view $B'$ as a finite subset of $F=F_1\ast\ldots\ast F_n$. It is clearly symmetric, contains the generators and the unit and $\sigma'$ is a partial pre-norm. Moreover, the norm $\sigma$ on $F$ determined by $\sigma'$ extends $\lambda_i$ for each $i\leq n$. Indeed, take any $i\leq n$ and $y\in F_i$. It follows from the definition that $\sigma(y)\leq \lambda_i(y)$. Suppose that $\sigma(y)<\lambda_i(y)$. Then there exists $y_1,\ldots,y_m\in B'$ such that $y=y_1\cdot\ldots\cdot y_m$ and $\sigma(y)=\sum_{j=1}^m \sigma'(y_j)$. For any $j\leq m$ if $y_j\notin A_i$ then set $\tilde{y}_j=1$, if $y_j\in A_i$ then let $\tilde{y}_j=y_j$. Since $y\in F_i$ we have that $y=\prod_{j=1}^m \tilde{y}_j$ and $$\lambda_i(y)\leq \sum_{j=1}^m \sigma'(\tilde{y}_j)\leq \sum_{j=1}^m \sigma'(y_j)=\sigma(y),$$ a contradiction.

However, $2\Gamma_i^j$ is not necessarily a MOC for every $x_{i,j}$ (and its inverse) anymore. That will be fixed in the next step.\\

\noindent {\bf Step 2.} Denote by $I$ the set $\{(i,j):i\leq n,j\leq n_i\}$. Then for every $(i,j)\in I$ and $r\in\Rea$ we set $\Gamma_{i,j}(r)=\Gamma_i^j+r$. Clearly, for every $r$, $\Gamma_{i,j}(r)\geq 2r$ and $\Gamma_{i,j}\leq 2\Gamma_i^j$.

Now denote by $\bar{\mathbb{W}}$ the alphabet $\{x_{i,j}^\varepsilon:(i,j)\in I,\varepsilon\in\{1,-1\}\}\cup\{1\}$. We recall that the elements of $F$ correspond to reduced words over the alphabet $\bar{\mathbb{W}}$.

Let now $w=w_1\ldots w_n$ be any word (not necessarily reduced) over $\bar{\mathbb{W}}$ and let $\rho$ be a match on $I=\{1,\ldots,n\}$ for $w$. Then we define the value $\lambda_\rho(w)$ by induction on $n$.

For technical reasons we also allow the case when $n=0$, i.e. $w$ is an empty word. Then we set $\lambda_\rho(w)=0$.

Suppose that $n=1$. Then the match $\rho$ is trivial and we set $\lambda_\rho(w)=\sigma(w)=\sigma(w_1)$.

Suppose now that $n>1$ and we have defined $\lambda_\rho(w)$ for every $w$ of length less than $n$ and every match $\rho$ for $w$. If $\rho(1)=n$ then $w=x_{i,j}^\varepsilon \tilde{w} x_{i,j}^{-\varepsilon}$ for some $(i,j)\in I$ and $\varepsilon\in\{1,-1\}$, where $\tilde{w}=w_2\ldots w_{n-1}$. By $\rho'$ we denote the match $\rho\upharpoonright [2,\ldots,n-1]$ for $\tilde{w}$ and we set $$\lambda_\rho(w)=\Gamma_{i,j}(\lambda_{\rho'}(\tilde{w})).$$

Suppose now that $\rho(1)\notin \{1,n\}$. Then denote by $\rho_1$ the match $\rho\upharpoonright [1,\ldots,\rho(1)]$ for $w_1\ldots w_{\rho(1)}$ and by $\rho_2$ the match $\rho\upharpoonright [\rho(1)+1,\ldots,n]$ for $w_{\rho(1)+1}\ldots w_n$. And we set $$\lambda_\rho(w)=\lambda_{\rho_1}(w_1\ldots w_{\rho(1)})+\lambda_{\rho_2}(w_{\rho(1)+1}\ldots w_n).$$

Finally, suppose that $\rho=\mathrm{id}_{\{1,\ldots,n\}}$. Then we set $\lambda_\rho(w)=\sigma(w')$, where, we recall, $w'$ is the reduced word obtained from $w$; i.e. an element of $F$.

We may now define the norm $\tilde{\lambda}$ as follows. For any $x\in F$ we set $$\tilde{\lambda}(x)=\min\{\lambda_\rho(w):w'=x,\rho\text{ is a match on }\{1,\ldots,|w|\}\text{ for }w\}\}.$$\\
Note that since $F$ and $\sigma$ are finitely generated we may indeed use the minimum in the formula above.

It follows from the definition that $\tilde{\lambda}$ is a norm. Indeed, clearly it is symmetric, since $\sigma$ was symmetric, and it vanishes only at $1$ since the minimum is used in the definition. Take now some $x,y\in F$. Let $w_x$ be a word satisfying $w_x'=x$ and $\rho_x$ a match for $w_x$ such that $\lambda_{\rho_x}(w_x)=\tilde{\lambda}(x)$. We also take $\rho_y$ and $w_y$ with analogous properties for $y$. Then we get that $$\tilde{\lambda}(x\cdot y)\leq \lambda_{\rho_x\cup\rho_y} (w_x w_y)=\lambda_{\rho_x}(w_x)+\lambda_{\rho_y}(w_y).$$

We now show that for each $(i,j)\in I$ and $\varepsilon\in\{1,-1\}$ and any $y\in F$ we have $\tilde{\lambda}(x_{i,j}^\varepsilon\cdot y\cdot x_{i,j}^{-\varepsilon})\leq \Gamma_{i,j}(\tilde{\lambda}(y))$. Let $w_y$ be a word satisfying $w_y'=y$ and $\rho_y$ a match for $w_y$ such that $\lambda_{\rho_y}(w_y)=\tilde{\lambda}(y)$. Suppose that $|w_y|=l$ and let $\rho$ be a match on $\{1,\ldots,l+2\}$, defined by $\rho(1)=l+2$, $\rho(l+2)=1$ and for any $1<i<l+2$, $\rho(i)=\rho_y(i-1)$, for the word $x_{i,j}^\varepsilon w_y x_{i,j}^{-\varepsilon}$. Then $$\tilde{\lambda}(x_{i,j}^\varepsilon\cdot y\cdot x_{i,j}^{-\varepsilon})\leq \lambda_\rho(x_{i,j}^\varepsilon w_y x_{i,j}^{-\varepsilon})=\Gamma_{i,j}(\tilde{\lambda}(y)).$$
Moreover, we claim that $\tilde{\lambda}$ still extends $\lambda_i$ on $F_i$ for each $i\leq n$. This is done completely analogously as we did it for $\sigma$. That is, for any $i\leq n$ and $x\in F_i$, if $\tilde{\lambda}(x)<\lambda_i(x)$, then there would be a word $w_x$ over $\bar{\mathbb{W}}$ and a match $\rho$ for $w_x$ such that $w_x'=x$ and $\lambda_\rho(w_x)<\lambda_i(x)$. Replace in $w_x$ each letter from $\bar{\mathbb{W}}\setminus \{x_{i,j}^\varepsilon:j\leq n_i,\varepsilon\in\{1,-1\}\}$  by $1$ and denote the obtained word $v_x$. Since $x\in F_i$ we still have that $v_x'=x$ and it follows directly from definition that $\lambda_i(x)\leq\lambda_\rho(v_x)\leq \lambda_\rho(w_x)$.\\

\noindent {\bf Step 3.} Now, for every $(i,j)\in I$, let $r_{i,j}$ be (the minimal number) such that $\Gamma_{i,j}(r_{i,j})\geq 2\lambda_i(x_{i,j})+r_{i,j}$. Set $r'=\max_{(i,j)\in I} r_{i,j}$ and $r=\max_{(i,j)\in I} \Gamma_{i,j}(r')$.
Since, as it it straightforwards to check, $\tilde{\lambda}$ is still proper, the set $Y=\{y\in F:\tilde{\lambda}(y)\leq r\}$ is finite.

Finally, we define a finitely generated norm $\lambda$ with the desired properties. We let $\lambda$ be generated by values of $\tilde{\lambda}$ on $B=B'\cup Y$, i.e. for any $x\in F$ we set $$\lambda(x)=\min\{\tilde{\lambda}(x_1)+\ldots+\tilde{\lambda}(x_m):x_1,\ldots,x_m\in B,x=x_1\cdot\ldots\cdot x_m\}.$$

Clearly, $\lambda$ extends $\lambda_i$ on $F_i$ since $\lambda_i$ was generated by $B_i$, $B_i\subseteq B'\subseteq B$ and $\tilde{\lambda}$ extends $\lambda_i$. Also, $\lambda$ coincides with $\tilde{\lambda}$ on $Y$.

And moreover, for any $(i,j)\in I$ and $\varepsilon\in\{1,-1\}$ and any $y\in F$ we have $$\lambda(x_{i,j}^\varepsilon\cdot y\cdot x_{i,j}^{-\varepsilon})\leq \Gamma_{i,j}(\lambda(y)).$$ Indeed, take any $(i,j)\in I$, $\varepsilon\in\{1,-1\}$ and $y\in F$. If $\lambda(y)>r'$ then $\Gamma_{i,j}(\lambda(y))\geq 2\lambda(x_{i,j})+\lambda(y)$. However, $\lambda(x_{i,j}^\varepsilon\cdot y\cdot x_{i,j}^{-\varepsilon})\leq 2\lambda(x_{i,j})+\lambda(y)$. Thus suppose that $\lambda(y)\leq r'$. Then $y\in Y$ and $\lambda(y)=\tilde{\lambda}(y)$. We have that $\tilde{\lambda}(x_{i,j}^\varepsilon\cdot y\cdot x_{i,j}^{-\varepsilon})\leq \Gamma_{i,j}(\tilde{\lambda}(y))\leq \Gamma_{i,j}(r')\leq r$. It follows that $x_{i,j}^\varepsilon\cdot y\cdot x_{i,j}^{-\varepsilon}\in Y$ and thus $\lambda(x_{i,j}^\varepsilon\cdot y\cdot x_{i,j}^{-\varepsilon})=\tilde{\lambda}(x_{i,j}^\varepsilon\cdot y\cdot x_{i,j}^{-\varepsilon})\leq \Gamma_{i,j}(\lambda(y))$. That finishes the proof.
\end{proof}
\begin{remark}
Matches were originally used by Ding and Gao in \cite{DiGao} for a convenient computation of the Graev bi-invariant metric. The same authors then used matches for constructing also continuous norms, or continuous left-invariant metrics in \cite{DiGa}, which is close to the approach we used in the previous proposition. The reader is invited to compare the construction in Step 2 in Proposition \ref{prop1} with the construction in Definition 3.3 in \cite{DiGa}. The same constructions were later used by Ding in \cite{Ding} to construct surjectively universal Polish groups. A reader familiar with these results will recognize that our construction is, in a sense, a generalization of those in \cite{DiGa}.
\end{remark}
The next proposition is a, sort of, \emph{metric residual finiteness} of normed free groups. It shows that normed free groups may be approximated by finite normed groups. That will be used in producing the desired sequence of finite normed groups from Theorems \ref{main1_restate} and \ref{main2}.
\begin{prop}\label{prop2}
Let $F$ be a finitely generated free group with a norm $\lambda$. Then for any finite subset $A\subseteq F$ containing the generators there exists a \emph{finite} group $H$ with a norm $\sigma$ and a partial monomorphism $\phi:A\subseteq F\hookrightarrow H$ which is also an isometry with respect to $\lambda$ and $\sigma$.

Moreover, if $\lambda$ is proper and for any free generator $x$ of $F$ some MOC $\Gamma_x$ of $x$ in $F$ is given, such that it is eventually greater than $2\lambda(x)+\mathrm{id}$, then, provided $A$ is large enough, $\Gamma_x$ remains a MOC for $\phi(x)$ in $H$.
\end{prop}
\begin{proof}
Let $M=\max\{\lambda(x):x\in A\}$, let $m=\min\{\lambda(x):x\in A\setminus \{1\}\}$ and let $K=\max\{|x|:x\in A\}$. Let $B=\{x\in F: |x|\leq K\cdot \lceil\frac{M}{m}\rceil\}$. Note that $A\subseteq B$. Since $F$ is residually finite there exists a finite group $H$ together with a partial monomorphism $\phi:B\subseteq F\hookrightarrow H$. Moreover, we may assume that $\phi[B]$ generates $H$. Note that then in fact $\phi[A]$ generates $H$ as $A$ contains the (free) generators of $F$.

To simplify the notation, for every $x\in B$ denote by $x'$ the element $\phi(x)\in H$. For every $x\in A$ set $\sigma'(x')=\lambda(x)$. Let $\sigma$ be a norm on $H$ generated by $\sigma'$. It suffices to prove that for every $x\in A$ we have $\sigma(x')=\sigma'(x')$ ($=\lambda(x)$).

Note that although $\sigma'$ is a partial norm on $\phi[A]$, it does not follow automatically that $\sigma$ extends $\sigma'$ as $H$ is not free, it is a finite group.

Suppose that for some $x\in A$ we have $\sigma(x')<\sigma'(x')$. Then there exist $x_1,\ldots,x_n\in A$ such that $x'=x'_1\cdot\ldots\cdot x'_n$ and $$\sigma(x')=\sum_{i=1}^n \sigma'(x'_i)<\lambda(x).$$ We claim that $n\leq \frac{M}{m}$. Indeed, we have $\sigma'(x')\leq M$ and for every $i\leq n$, $\sigma'(x'_i)\geq m$. Thus if $n>\frac{M}{m}$, then $\sum_{i=1}^n \sigma'(x'_i)> m\cdot \frac{M}{m}> M$, a contradiction.

Moreover, for each $i\leq n$ we have $|x_i|\leq K$. Thus $|x_1\cdot\ldots\cdot x_n|\leq K\cdot\frac{M}{n}$. Consequently, $x_1\cdot\ldots x_n$ is in $B$, so in the domain of $\phi$. However, then it follows that $x=x_1\cdot\ldots\cdot x_n$ as $\phi$ is a partial monomorphism. But we have $$\lambda(x)\leq \lambda(x_1)+\ldots+\lambda(x_n)=\sigma'(x'_1)+\ldots+\sigma'(x'_n)=\sigma(x'),$$ a contradiction.\\

It remains to prove the `moreover' part from the statement of the proposition. Suppose that $(F,\lambda)$ is such that $\lambda$ is proper, e.g. $\lambda$ is finitely generated. Take some generator $x\in F$ and let $\Gamma_x$ be a MOC for $x$ in $F$ such that there is some $r'$ such that $\Gamma_x(r'')\geq 2\lambda(x)+r''$ for $r''\geq r'$. Set $r=\Gamma_x(r')$. Then the set $B=\{y\in F:\lambda(y)\leq r\}$ is finite. Suppose now that $(H,\sigma)$ is a finite normed group and $\phi:A\subseteq F\rightarrow H$ a partial monomorphism on some finite set $A$ containing $B$ which is isometric. Then we claim that $\Gamma_x$ is a MOC for $\phi(x)$ in $H$. Indeed, take some $y\in H$. If $\sigma(y)>r'$ then $\Gamma_x(\sigma(y))\geq 2\sigma(\phi(x))+\sigma(y)\geq \sigma(\phi(x)^{-1}\cdot y\cdot\phi(x))$. If $\sigma(y)\leq r'$ then $y=\phi(y')$ for some $y'\in B$ and $\lambda(y')=\sigma(y)$. Since $\lambda(x^{-1}\cdot y'\cdot x)\leq \Gamma_x(\lambda(y'))\leq r$ we have $x^{-1}\cdot y'\cdot x\in B$, thus $\sigma(\phi(x)^{-1}\cdot y\cdot \phi(x))=\lambda(x^{-1}\cdot y'\cdot x)\leq \Gamma_x(\lambda(y'))=\Gamma_x(\sigma(y))$.
\end{proof}

\begin{constr}\label{main_constr}
Let $\{(F_n,\nu_n):n\in\Nat\}$ be an enumeration of all finitely generated free groups with rational finitely generated norms, i.e. norms taking values in the rationals. We shall denote the generators of $F_i$ by $x_{i,1},\ldots,x_{i,n_i}$, for each $i$. For each $n\in\Nat$ we use Proposition \ref{prop1} to define a norm $\lambda_n$ on $G_n=F_1\ast\ldots\ast F_n$ which extends $\nu_i$ for $i\leq n$, and moreover, for each $i\leq n$, $j\leq n_i$ we have that $2\Gamma_i^j$ is a modulus of continuity of $x_{i,j}$ in $G_n$ (where $\Gamma_i^j$ was the minimal MOC for $x_i^j$ in $F_i$). Suppose that $\lambda_n$ is generated by some $\lambda'_n$ defined on a finite set $A_n\subseteq G_n$. Set $k_n=\max_{x\in A_n} |x|$ and let $B_n=\{x\in G_n: |x|\leq k_n\}$. We use Proposition \ref{prop2} to get a finite group $H_n$ with a norm $\rho_n$ such that there is a partial monomorphism $\phi_n:B_n\hookrightarrow H_n$ which is isometric with respect to $\lambda_n$ and $\rho_n$, and moreover, for every generator $x_{i,j}$, $i\leq n$, $j\leq n_i$, $\Gamma_{i,j}=2\Gamma_i^j$ is a MOC for $\phi_n(x_{i,j})$ in $H_n$.

Finally, consider any non-principal ultrafilter on $\Nat$ and set $\mathbb{G}$ to be the corresponding metric ultraproduct of the sequence $(H_n,\rho_n)_n$.\\
\end{constr}
\begin{thm}\label{final_thm}
$\mathbb{G}$ contains isometrically every separable normed topological group.
\end{thm}
\begin{remark}
Theorem \ref{final_thm} covers Theorem \ref{main2} from the introduction. Theorem \ref{main1}, resp. \ref{main1_restate} will follow by a rather standard argument which we shall provide after the proof of Theorem \ref{final_thm}.
\end{remark}
\begin{proof}
 Let $(E,\rho)$ be an arbitrary separable normed group. Let $(e_n)_n$ be an infinite set of generators such that the Ra\v ikov metric completion of the subgroup generated by $(e_n)_n$ contains $E$. By Observation \ref{obs1}, we may suppose that the subgroup generated by $(e_n)_n$ is free if we view $\rho$ as a seminorm. For any $x\in E$ by $\Gamma_x$ we shall denote $\Gamma_x^E$, i.e. the minimal modulus of continuity for $x$ in $E$.

For any $n$, let $E_n$ be the free group freely generated by $e_1,\ldots,e_n$. Let $C_n$ be the set $\{x\in E_n:|x|\leq n\}$.

We define a rational partial \underline{norm} (note that $\rho$ in contrast may be just a seminorm) $\sigma'_n$ on $C_n$. We take as $\sigma'_n$ any rational partial norm $\sigma'$ on $E_n$ with the property that for every $w\in C_n$ we have
\begin{multline}\label{ratnorm_eq}
\sigma'(w)\geq \rho(w),\\
\sigma(w)-\rho(w_h)\leq 1/m,
\end{multline}
 where $m=|C_n|$.
\begin{claim}
Such a rational partial norm $\sigma$ exists.
\end{claim}
To show it enumerate $C_n$ as $c_1,\ldots,c_m$ in such a way that $\rho(c_1)\geq \rho(c_2)\geq\ldots\geq \rho(c_m)$. Let $C_{min}=\min\{1/m,\min\{|\rho(c_i)-\rho(c_j)|:i,j\leq m,\rho(c_i)\neq \rho(c_j)\}\}$. Let $(\delta_i)_{i=1}^{2m+1}$ be an increasing sequence of positive real numbers such that for each $i\leq 2m+1$
\begin{itemize}
\item $\delta_i<C_{min}$,
\item if for some $i\neq j$, $\rho(c_i)=\rho(c_j)$, then $\delta_i=\delta_j$,
\item $\rho(c_i)+\delta_i\in \Rat$.
\end{itemize}

Then for $c_i\neq 1$ we set $\sigma'(c_i)=\rho(c_i)+\delta_i$, and $\sigma'(1)=0$. Clearly, it is rational, and it is symmetric since if $c_i=c^{-1}_j$ then $\delta_i=\delta_j$, thus $\sigma'(c_i)=\rho(c_i)+\delta_i=\rho(c_j)+\delta_j=\sigma'(c_j)$. Let $c_i=c_{i_1}\cdot\ldots\cdot c_{i_j}$. If for any $l\leq j$ we have $i_l\geq i$, i.e. $\rho(c_i)\leq \rho(c_{i_l})$, then we have $$\sigma'(c_i)=\rho(c_i)+\delta_i\leq \rho(c_i)+(\rho(c_{i_l})-\rho(c_i))\leq$$  $$\rho(c_{i_l})+\delta_{i_l}=\sigma'(c_{i_l})\leq \sum_{l=1}^j \sigma'(c_{i_l}).$$ 
If for every $l\leq j$ we have $\rho(c_i)>\rho(c_{i_l})$, then for every $l\leq j$ we have $\delta_i\leq \delta_{i_l}$ and thus $$\sigma'(c_i)=\rho(c_i)+\delta_i\leq \sum_{l=1}^j \rho(c_{i_l})+\delta_{i_l}=\sum_{l=1}^j \sigma'(c_{i_l}).$$
This proves the claim.\\

We set $\sigma_n$ to be the (rational finitely generated) norm on $E_n$ generated by $\sigma'_n$.

For each $n$, there is $i(n)$ such that $(E_n,\sigma_n)$ is equal to $(F_{i(n)},\nu_{i(n)})$. It follows that we can find a strictly increasing sequence of natural numbers $i_1<i_2<\ldots$ such that for each $k\in\Nat$ and every $i_k\leq l<i_{k+1}$, $G_l$ contains $F_{i(k)}=E_k$ as a subgroup. Thus for every $n$, $m\geq n$ and $i_m\leq l<i_{m+1}$ fix some isometric monomorphism $\phi: (E_m,\sigma_m)\rightarrow G_l$ and denote by $e_n^l$ the element $\phi(e_n)\in G_l$. For $l<i_n$, set $e_n^l=1$. So we have defined elements $e_n^l$ for all $n,l\in\Nat$.\\

Now notice that by \eqref{ratnorm_eq}, for any $w\in E$ we have
\begin{equation}\label{ratnorm_eq2}
\rho(w)=\lim_n \sigma_n(w).
\end{equation}
For any $n\leq m\in \Nat$, let $\Gamma_{e_n}^m$ be the minimal MOC for $e_n$ in $(E_m,\sigma_m)$ and $\Gamma_{e_n}$ the minimal MOC for $e_n$ in $(E,\rho)$. It follows from \eqref{ratnorm_eq} and \eqref{ratnorm_eq2} that
\begin{equation}\label{MOC_eq}
\Gamma_{e_n}=\lim_m \Gamma_{e_n}^m,
\end{equation}
 i.e. for any $r$, $\Gamma_{e_n}(r)=\lim_m \Gamma_{e_n}^m(r)$. By Proposition \ref{prop1}, we have

\begin{equation}\label{MOC_eq2}
\Gamma_{e_n^l}^{G_l}\leq 2\Gamma_{e_n}^m\quad \forall n \; \forall m\geq n\; \forall i_m\leq l<i_{m+1}.
\end{equation}
Thus, if we denote by $f_n^m$ the element $\phi_m(e_n^m)$ in $H_m$, for all $n,m\in\Nat$, we still have that $\Gamma_{f_n^m}^{H_m}\leq 2\Gamma_{e_n}^m$ (by Proposition \ref{prop2}). Recall that $\phi_m$ is the partial isometric monomorphism from Construction \ref{main_constr} (where it was obtained using Proposition \ref{prop2}).

For each $n$ consider the sequence $(f_n^m)_m$. By \eqref{MOC_eq} and \eqref{MOC_eq2}, the elements $(f_n^m)_m$ are bounded by a common MOC, thus $(f_n^m)_m$ is continuous in the ultraproduct and belongs to $\mathbb{G}$. We shall denote the corresponding element there by $g_n$.

We now claim that $\langle g_n:n\in\Nat\rangle\leq \mathbb{G}$ is isometrically isomorphic to $\langle e_n:n\in\Nat\rangle\leq E$. For each $n\in\Nat$ and $w\in E_n$, denote by $w_{\mathbb{G}}$ the corresponding element in $\langle g_n:n\in\Nat\rangle$, i.e. an element obtained by a canonical evaluation where $e_n$ is evaluated as $g_n$. Similarly, for all $m\geq i_n$ denote by $w_m$ the evaluation of $w$ in $\langle f_i^m:i\leq n\rangle \leq H_m$. Then for any $n$ and $w\in E_n$ we have $$\lambda(w_{\mathbb{G}})=\lim_{\mathcal{U}} \lambda_m (w_m)=\lim_{m\to\infty} \lambda_m(w_m)=\rho(w).$$

Since $\mathbb{G}$ is Ra\v ikov metrically complete, it contains isometrically $E$.
\end{proof}
\begin{proof}[Proof of Theorem \ref{main1_restate}]
We claim that the sequence $(H_n,\rho_n)_n$ from Construction \ref{main_constr} is as desired. Fix some normed topological group $(G,\lambda)$, some finite subset $F\subseteq H$ and some $\varepsilon>0$. We may without loss of generality suppose that $G$ is separable; otherwise we could replace $G$ by some separable subgroup of $G$ containing $F$. Suppose, to reach a contradiction, that there is an infinite subset $A\subseteq \Nat$ such that for all $i\in A$ there is no $\varepsilon$-homomorphism from $F$ into $G_i$. Let $\mathcal{U}$ be an arbitrary non-principal ultrafilter on $\Nat$ such that $A\in\mathcal{U}$. By Theorem \ref{final_thm}, the metric ultraproduct of the sequence $(G_n)_n$ using $\mathcal{U}$ contains $G$ isometrically. Moreover, it follows from the proof of Theorem \ref{final_thm} that if we choose some generating sequence $(e_n)_n$ of $G$ so that it contains the elements of $F$, then we obtain an isometric embedding $\psi:G\rightarrow \mathbb{G}$, where $\mathbb{G}$ is the ultraproduct, such that for every $f\in F$, $$\Gamma_{\psi(f)}^\mathbb{G}\leq 2\Gamma_f^G.$$ As usual, we shall suppose that for each $g\in G$, $\psi(g)$ is a sequence from $\prod_n G_n$ rather than some equivalence class, and for each $i\in\Nat$, by $\psi(g)_i$ we denote the corresponding projection on the $i$-th coordinate.

Then by a standard ultraproduct argument (essentially by the classical \L o\' s theorem) there exists a set $B\in\mathcal{U}$ such that for all $i\in B$ the map $\phi_i:F\rightarrow G_i$ defined by $f\to \psi(f)_i$ is an $\varepsilon$-homomorphism such that moreover $\Gamma_{\phi_i(f)}^{G_i}\leq 2\Gamma_f^G+\varepsilon\mathrm{id}$. Taking any $i\in A\cap B$ leads to a contradiction.
\end{proof}
Let us conclude with few problems. First, we want to ask whether the analogous result holds in the category of groups with bi-invariant metric. Thus we want to ask whether not only every discrete group is weakly sofic, which is the problem of Glebsky and Rivera, but whether actually every group with bi-invariant metric is weakly sofic.
\begin{question}\label{ques1}
Does every group with bi-invariant metric isometrically embed into a metric ultraproduct of finite groups with bi-invariant metric?
\end{question}

Let us offer also a weakening of the previous question. As weakly sofic groups generalize sofic groups, one can generalize the notion of hyperlinear groups by defining \emph{weakly hyperlinear groups} as those groups that can be approximated by compact groups with bi-invariant metric, or equivalently, as those groups that embed as subgroups into metric ultraproducts of compact groups with bi-invariant metric. This notion was introduced by Jakub Gismatulin. Clearly, the notion of weakly hyperlinear groups makes again sense also for metric groups with bi-invariant metric.
\begin{question}
Is every group with bi-invariant metric weakly hyperlinear?
\end{question}
\begin{remark}
During the review process of the paper, Question \ref{ques1} was answered negatively. Nikolov, Schneider and Thom in \cite{NST} prove that no compact connected non-abelian Lie group with a compatible bi-invariant metric embeds into a metric ultraproduct of finite groups with bi-invariant metric.
\end{remark}

\end{document}